\documentclass[12pt,reqno,a4paper]{amsart}
\usepackage{extsizes}
\usepackage{blindtext}
\usepackage{fullpage}
\usepackage{mathtools}
\usepackage{amsmath,amssymb,amsthm}
\usepackage{amscd}
\usepackage{bm}
\usepackage{hyperref}
\usepackage{dsfont}
\usepackage{enumerate}
\usepackage{epsfig}
\usepackage{float, graphicx}
\usepackage{latexsym, amsxtra}
\usepackage{mathrsfs}
\usepackage{multicol}
\usepackage[normalem]{ulem}
\usepackage{psfrag}
\usepackage[parfill]{parskip}
\usepackage{stmaryrd}
\usepackage{tikz}
\usepackage[T1]{fontenc}
\usepackage{url}
\usepackage{verbatim}
\usepackage{indentfirst}
\usepackage{tikz-cd}

\usepackage{mathtools}

\makeatletter
\def\thm@space@setup{%
  \thm@preskip=2ex \thm@postskip=2ex
}
\makeatother

\hypersetup{hidelinks}

\newtheorem{thm}{Theorem~}[section]
\newtheorem{lem}[thm]{Lemma~}

\newtheorem{prop}[thm]{Proposition~}

\newtheorem{de}[thm]{Definition~}
\newtheorem{rmk}[thm]{Remark~}
\newtheorem{ex}[thm]{Example~}

\newcommand{\CC}{\mathbb{C}}
\newcommand{\ZZ}{\mathbb{Z}}

\newcommand{\PP}{\mathbb{P}}

\newcommand{\QQ}{\mathbb{Q}}
\newcommand{\DD}{\mathbb{D}}
\newcommand{\HH}{\mathbb{H}}

\newcommand{\Prd}{\mathscr{P}}
\newcommand{\F}{\mathcal{F}}

\newcommand{\calH}{\mathcal{H}}

\newcommand{\calM}{\mathcal{M}}
\newcommand{\calV}{\mathcal{V}}
\newcommand{\calO}{\mathcal{O}}
\newcommand{\calS}{\mathcal{S}}

\newcommand\ADE{\mathrm{ADE}}
\newcommand\GIT{\mathrm{GIT}}

\newcommand\Sym{\mathrm{Sym}}
\newcommand\SL{\mathrm{SL}}

\newcommand\IV{\mathrm{IV}}

\newcommand\PSL{\mathrm{PSL}}

\newcommand{\Tor}{\mathrm{Tor}}
\newcommand{\rank}{\mathrm{rank}}
\newcommand{\Pic}{\mathrm{Pic}}
\newcommand{\diag}{\mathrm{diag}}

\newcommand{\GL}{\mathrm{GL}}
\newcommand{\I}{\mathrm{I}}

\newcommand{\bs}{\backslash}
\newcommand{\dbs}{\bs\!\! \bs}

\title{Moduli of Nodal Sextic Curves via Periods of $K3$ Surfaces}
\vspace{1.2cm}
 \author{Chenglong Yu, Zhiwei Zheng}
 \date{}

 \newcommand{\Addresses}{{
  \bigskip
  \footnotesize

  C.~Yu, \textsc{Tsinghua University, Beijing, China} \par\nopagebreak
   \textit{E-mail address}: \texttt{yuchenglong@tsinghua.edu.cn}

  \medskip

  Z.~Zheng, \textsc{Tsinghua University, Beijing, China}\par\nopagebreak
  \textit{E-mail address}: \texttt{zhengzhiwei@mail.tsinghua.edu.cn}
}}

\begin{document}
\bibliographystyle{amsalpha}

\begin{abstract}
In this paper we study the moduli spaces of nodal sextic curves. We realize each irreducible component of the GIT space of sextic curves with given number of nodes as an open subspace of type $\IV$ arithmetic quotients. We then focus on the compactifications of the moduli spaces, one side is the geometric ($\GIT$) compactifications, the other side is the Hodge theoretic compactifications such as Looijenga compactifications and Baily-Borel compactifications. The main result is the isomorphism between GIT and Looijenga compactifications. Some examples are closely related to del Pezzo surfaces. We also extend our results to moduli of nodal sextic curves with specified symmetry. 
\end{abstract}

\maketitle


\section{Introduction}
\label{section: introduction}
The study of moduli spaces of nodal plane curves has a long history tracing back to Severi \cite{severi1968curve}. He proved the smoothness and counted the dimension of the moduli space of nodal curves with given degree and number of nodes. (He also asserted the connectedness of the moduli of irreducible nodal plane curves with fixed degree and number of nodes, which was proved by Harris \cite{harris1986severi}.)

In this paper we focus on sextic curves. The double cover of $\PP^2$ branched along a smooth plane sextic is a $K3$ surface of degree $2$ and a generic $K3$ surface of degree $2$ arises this way. Consequently the period map of $K3$ surfaces gives rise to an open embedding from the $\GIT$ moduli of smooth sextic curves to an arithmetic quotient of type $\IV$ domain of dimension $19$.  An important problem related to this construction is the compactification of the moduli spaces. On one side, we can obtain a geometric ($\GIT$) compactification parametrizing certain singular curves. On the other side, we have a Hodge theoretical (Baily-Borel, or toroidal) compactification describing the degenerations of Hodge structures. A natural question is the comparison between the two different approaches of compactifications. The remarkable work by Shah \cite{shah1980completek3} realizes the $\GIT$ compactification as an explicit modification of the Baily-Borel compactification. This construction fits into a more general framework developed by Looijenga \cite{looijenga2003ball,looijenga2003typefour}, which we explain below.

The period map extends to a morphism from the moduli space of plane sextic curves with at worst $\ADE$ singularities to the arithmetic quotient, with image the complement of an irreducible divisor which is the quotient of a hyperplane arrangement $\calH_{\infty}$ in the type $\IV$ domain (see theorem \ref{theorem: shah}). Looijenga \cite{looijenga2003ball,looijenga2003typefour} has developed a machinery for compactifications of arithmetic quotients of arrangement complements in complex hyperbolic balls or type $\IV$ domains. Notice that these are the only irreducible Hermitian symmetric domains which have totally geodesic hypersurfaces. There are two steps in Looijenga compactification. First step is a partial blowup of the boundary components of Baily-Borel compactification, sitting between toroidal compactification and Baily-Borel. Second step is blowups of the intersection poset of the hyperplane arrangement and blowdowns in the opposite direction. 

In \cite{zarhin1983hodge}, Zarhin classified the Mumford-Tate groups of $K3$ type Hodge structures. The corresponding Mumford-Tate subdomains are complex hyperbolic balls and type $\IV$ domains.  In particular, the moduli spaces of the $K3$ type Hodge structures with specified symmetries give rise to arithmetic quotients of complex hyperbolic balls and type $\IV$ domains.  Symmetries of Hodge structures can arise from those of the geometric objects. This leads to constructions of many moduli spaces as arithmetic quotients of complex hyperbolic balls or type IV domains, such as Kond\=o \cite{kondo2000complex, kondo2000moduli}, Allcock-Toledo-Carlson \cite{allcock2002complex, allcock2011moduli}, Looijenga-Swierstra \cite{looijenga2007period}, Laza-Pearlstein-Zhang \cite{laza2017moduli}. See also \cite{yu2018moduli} for a systematic treatment of symmetric cubic fourfolds.

The starting idea of this work is that symmetries of Hodge structures can also arise from degenerations of geometric objects. Precisely, given a one dimensional degeneration of smooth projective varieties to a nodal one, we have a Picard-Lefschetz reflection on the cohomology of a generic fiber with respect to the vanishing cycle. This reflection can be regarded as a symmetry of the Hodge structure of a generic fiber. This inspires us to characterize moduli spaces of degenerated objects using modified period maps. In this paper, we work on the case of degenerated sextic curves.

A double cover of $\PP^2$ branched along a nodal sextic curve is a nodal $K3$ surface, with nodes at the preimage of singularities on that curve. After resolving the singularities, we obtain a $K3$ surface with natural lattice polarization (we refer to \cite{dolgachev1996mirror} for the notion of lattice polarization). Along this direction, there are lots of works on moduli spaces of singular sextic curves, especially when the sextic curve is the union of curves with lower degrees. In \cite{matsumoto1992sixlines}, Matsumoto-Sasaki-Yoshida studied the moduli space of (ordered) six lines on $\PP^2$. They realized the $\GIT$ compactification of the moduli space as the Baily-Borel compactification of an arithmetic quotient of type IV domain of dimension 4. It is called Coble variety which is a double cover of $\PP^4$ branched along the Igusa quartic given by equation $(x_0 x_1+x_0 x_2+x_1 x_2-x_3 x_4)^2-4x_0x_1x_2 \sum x_i=0$, see \cite[Example 7]{hunt2000nice}. In \cite{laza2009deformation}, Laza studied the moduli of pairs consisting of a plane quintic curve and a line. An important motivation for Laza's work is the study of the deformations of the minimally elliptic surface singularity $N_{16}$. Recently, Pearlstein and Zhang \cite{pearlstein2019horikawa} related this period map to the study of special Horikawa surfaces and proved generic Torelli theorem. Another case about triples consisting of a plane quartic curve and two lines is studied by Gallardo, Martinez-Garcia and Zhang in \cite{gallardo2017compactifications}. In the  works \cite{matsumoto1992sixlines}, \cite{laza2009deformation}, \cite{gallardo2017compactifications} mentioned above, the GIT compactifications of moduli spaces with respect to suitable polarizations are isomorphic to the Baily-Borel compactifications of the period domains. The main purpose of this paper is to give a uniform and systematic account of these results, hence enlarge the list of examples for which we have identifications between $\GIT$ compactificaitions and Hodge theoretic (Looijenga, most cases Baily-Borel) compactifications.

Notice that there is a natural stratification on the moduli space of nodal plane sextic curves induced by number of nodes. Let $T$ be a singular type, which corresponds to an irreducible component of certain strata. Let $\F_T$ be the moduli space of sextic curves of type $T$ and $\overline{\F}_T$ a compactification constructed via $\GIT$. See \S \ref{subsection: GIT of singular sextic curves} for detailed definition.

We denote by $n$ the number of nodes on a generic sextic curve $C$ of singular type $T$. Let $S$ be the double cover of $\PP^2$ branched along $C$, and $\widetilde{S}\longrightarrow S$ the resolution of the $n$ nodes on $S$. The resolution gives $n$ smooth rational curves (with self-intersection $-2$) on $\widetilde{S}$, while the total transformation of the degree $2$ polarization on $S$ is a norm $2$ class in the Picard lattice $\Pic(\widetilde{S})$. Those $n+1$ classes generate a sublattice (of type $\I_{1, n}(2)$, and not primitive when $C$ has multiple irreducible components) in $\Pic(\widetilde{S})$. In general it is an interesting question to determine the Picard lattice $\Pic(\widetilde{S})$. This question was answered case by case in previous works. The first new result in our work is a uniform treatment and characterization of $\Pic(\widetilde{S})$ for all singular types $T$ of nodal sextic curves. We process in two steps. Firstly, by purely topological argument, we show that (essentially in lemma \ref{lemma: invariant cohomology}) the $n$ $(-2)$-classes and the norm $2$ class generate the rational Picard group $\Pic(\widetilde{S})_{\QQ}=H^2(\widetilde{S}, \QQ)\cap H^{1,1}$ (notice that $C$ is chosen generically). The second step is to describe how the lattice $\I_{1, n}(2)$ is saturated to $\Pic(\widetilde{S})$. This is done in proposition \ref{proposition: m} where we show that the quotient of $\Pic(\widetilde{S})$ by $\I_{1, n}(2)$ is $(\ZZ/2)^{l-1}$, where $l$ is the number of irreducible components of $C$.

Denote by $M$ a lattice isomorphic to $\Pic(\widetilde{S})$ for a generic choice of $C$ of singular type $T$. Then the $K3$ surface associated with any point in $\F_T$ is naturally $M$-polarized. The period map of the $M$-poarized $K3$ surfaces gives a morphism $\Prd\colon \F_T\longrightarrow \Gamma\bs\DD$. Here $\Gamma\bs\DD$ is the global period domain for $M$-polarized $K3$ surfaces. There is a $\Gamma$-invariant hyperplane arrangement $\calH_*$ in $\DD$. We have the Looijenga compactification $\overline{\Gamma\bs\DD}^{\calH_*}$ of $\Gamma\bs(\DD-\calH_*)$, see \cite{looijenga2003typefour}. Our main theorem is:
\begin{thm}[Main Theorem]
\label{theorem: main}
For any singular type $T$, the period map $\Prd\colon \F_T\longrightarrow \Gamma\bs \DD$ is an algebraic open embedding with image $\Prd(\F_T)\subset \Gamma\bs (\DD-\calH_*)$, and the map $\Prd\colon \F_T\to \Gamma\bs (\DD-\calH_*)$ extends to an isomorphism between the GIT compactification and the Looijenga compactification $\Prd\colon \overline{\F}_T\cong \overline{\Gamma\bs \DD}^{\calH_*}$.
\end{thm}

We briefly introduce the strategy of the proof. The injectivity of the period map essentially relies on the global Torelli for $K3$ surfaces (however, some subtlety on lattice-theoretic side needs to be dealt with). We then show that the period map is a bimeromorphism by observing that both two sides are irreducible varieties of the same dimension. The key step is to establish the identification between $\GIT$ compactification and Looijenga compactification. We use the ideas and techniques developed in \cite{yu2018moduli}. Namely, we observe the following commutative diagram:
\begin{equation*}
\begin{tikzcd}
\F_T\arrow{r}{\Prd}\arrow[hook]{d}     & \Gamma\bs(\DD-\calH_*)\arrow[hook]{d} \\
\overline{\F_T}\arrow{d}{j}  & \overline{\Gamma\bs\DD}^{\calH_*}\arrow{d}{\pi}\\
\overline{\calM}\arrow{r}{\Prd} & \overline{\widehat{\Gamma}\bs\widehat{\DD}}^{\calH_{\infty}}.
\end{tikzcd}
\end{equation*}
Here $\overline{\calM}$ and $\overline{\widehat{\Gamma}\bs\widehat{\DD}}^{\calH_{\infty}}$ are the GIT compactification and the Looijenga compactification for the moduli space of smooth sextic curves, see \S\ref{section: review period of degree two K3}. Both $\overline{\F_T}$ and $\overline{\Gamma\bs\DD}^{\calH_*}$ are normal projective varieties. The morphism $j\colon\overline{\F_T}\longrightarrow \overline{M}$ is naturally finite. From the appendix in \cite{yu2018moduli}, the morphism $\pi$ between Looijenga compactifications is also finite. Combining with the fact that the period map on the top row of the diagram is a bimeromorphism, we conclude the existence of a natural identification between $\overline{\F_T}$ and $\overline{\Gamma\bs\DD}^{\calH_*}$.

\begin{rmk}
In this paper we only consider nodal sextic curves. The more general case for sextic curves with $\ADE$ singularities is also interesting but more involved. In particular, Galati \cite{galati2009cusps} found an interesting phenomenon for cuspidal sextic curves. She showed that the moduli space of irreducible sextic curves with $6$ cusps has two irreducible components and each has dimension equal to $7$. Another interesting example related is the Ap\'ery pencil, which is a one dimensional family of sextic curves with certain type of $\ADE$ singularity and symmetry. This pencil appeared in Ap\`ery's proof of irrationality of $\zeta(3)$ in \cite{apery1979irrationalite}. Peters and Beukers \cite{beukers1984afamily} observed that this family is related to $K3$ surfaces polarized by certain lattice of rank $19$. We will return to $\ADE$ situation in future works. 
\end{rmk}

\noindent\textbf{Structure of the paper:}
In \S\ref{section: review period of degree two K3} we introduce the work by Shah and Looijenga on moduli of $K3$ surfaces of degree $2$. In  \S\ref{section: period map} we formulate and prove our main theorem \ref{theorem: main}. In the end of this main section we include a criterion on $\GIT$ side about when $\overline{\Gamma\bs\DD}^{\calH_*}$ is actually a Baily-Borel compactification. Section \ref{section: example} is devoted to examples and applications. Irreducible nodal sextic curves are naturally related with del Pezzo surfaces. We investigate this relation in \S \ref{section: del pezzo}. When a generic sextic curve of singular type $T$ is a union of smooth components, we identify (in \S \ref{subsection: union}) $\overline{\F}_T$ with $\GIT$ quotients of products of projective spaces. See equation \eqref{equation: GITspaces} and table \eqref{table: smooth components}. In our previous work \cite{yu2018moduli}, we considered moduli spaces of symmetric cubic fourfolds. Similar construction is achieved for sextic curves in \S\ref{section: with symmetry}. We end the paper with two examples \ref{example: six involution} and \ref{example: line quintic three} for which the sextic curves are both nodal and symmetric. Interestingly, example \ref{example: line quintic three} gives rise to a ball quotient of dimension $5$, and is related to one of the examples in Deligne-Mostow theory \cite{deligne1986monodromy}.

\noindent \textbf{Acknowledgement:}
The first author is supported by the Simons Collaboration on Homological Mirror Symmetry 2015-2019. He thanks Mao Sheng, Colleen Robles, Matt Kerr, Bong Lian and Shinobu Hosono for helpful discussion about the paper \cite{matsumoto1992sixlines} on moduli of six lines. He also thanks his advisor, Prof. S.-T. Yau for his constant support. The second author thanks his advisor, Eduard Looijenga, for his help and encouragement along the way.

\section{Periods of K3 Surfaces of Degree $2$}
\label{section: review period of degree two K3}
In this section we review the characterization of moduli of sextic curves via period map of $K3$ surfaces. See \cite{shah1980completek3}, \cite[\S8, Theorem 8.6]{looijenga2003typefour}, \cite[\S1.2.3]{laza2016persepectives}.

Let $V$ be a 3-dimensional vector space over $\CC$. We call a sextic curve together with an embedding into $\PP(V)$ a plane sextic curve. The space of plane sextic curves is $\PP\Sym^6(V^*)$. We call a homogeneous polynomial smooth if it defines a smooth hypersurface. Denote by $\PP\Sym^6(V^*)^{sm}$ the subset of $\PP\Sym^6(V^*)$ consisting of smooth sextic polynomials. Consider the action of $\SL(V)$ on $\PP \Sym^6(V^*)$, and define $\calM$ to be the $\GIT$ quotient $\SL(V)\dbs \PP\Sym^6(V^*)^{sm}$. Denote by $\overline{\calM}$ the $\GIT$ compactification of $\calM$, and $\calM_1$ the moduli of sextic curves with at worst simple singularieties.

Consider the double cover $S_C\longrightarrow \PP^2$ branched along a smooth plane sextic curve $C\subset \PP^2$. Let $H\in H^2(S_C,\ZZ)$ be the pull-back of the hyperplane class of $\PP^2$. Then $(S_C, H)$ is a polarized smooth $K3$ surface with $\varphi(H, H)=2$. Here $\varphi$ is the topological intersection pairing on the second cohomology.

The isomorphism type of $(H^2(S_C,\ZZ), \varphi, H)$ does not depend on $C$. Let $(\Lambda, \varphi, H)$ be a triple isomorphic to $(H^2(S_C,\ZZ), \varphi, H)$. Then $(\Lambda, \varphi)\cong U^3\oplus E_8(-1)^2$ is an even unimodular lattice of signature $(3, 19)$. This is usually called the $K3$ lattice. We write $\Lambda$ for $(\Lambda, \varphi)$ for short. Let $\Lambda_0$ be the orthogonal complement of $H$ in $\Lambda$. Let $\widehat{\DD}$ be a component of $\PP\{x\in \Lambda_0\otimes \CC|\varphi(x, x)=0, \varphi(x,\overline{x})>0\}$. This is the period domain for $K3$ surfaces of degree $2$.

The second cohomology of degree two $K3$ family over $\PP\Sym^6(V^*)^{sm}$ gives rise to a variation of Hodge structures on $\PP\Sym^6(V^*)^{sm}$. This induces a period map $\Prd\colon \PP\Sym^6(V^*)^{sm}\to \widehat{\Gamma}\bs \widehat{\DD}$. Here $\widehat{\Gamma}\subset O(\Lambda, \varphi, H)$ is the index two subgroup leaving $\widehat{\DD}$ stable. The quotient space $\widehat{\Gamma}\bs \widehat{\DD}$ is an arithmetic quotient of type $\IV$ domain, hence quasi-projective thanks to the Baily-Borel compactification (see \cite{borel1966}). Since $\SL(V)$ is a connected Lie group acting on the $K3$ family, the holomorphic map $\Prd$ descends to
\begin{equation}
\label{equation: period of sextic curve}
\Prd \colon \SL(V)\dbs \PP\Sym^6(V^*)^{sm}\to \widehat{\Gamma}\bs \widehat{\DD}.
\end{equation}

We have two $\widehat{\Gamma}$-invariant hyperplane arrangements $\calH_{\Delta}$ and $\calH_{\infty}$ in $\widehat{\DD}$, which correspond to two $O(\Lambda,\varphi, H)$-orbits of vectors in $\Lambda_0$ with self-intersection $-2$. The vectors defining $\calH_{\Delta}$ have divisibility $1$, and those defining $\calH_{\infty}$ have divisibility $2$. We have the Looijenga compactification of $\widehat{\Gamma}\bs (\widehat{\DD}-\calH_{\infty})$, denoted by $\overline{\widehat{\Gamma}\bs \widehat{\DD}}^{\calH_{\infty}}$. See \cite{looijenga2003typefour}.

From \cite{shah1980completek3} and \cite{looijenga2003typefour}, we have
\begin{thm}[Shah, Looijenga]
\label{theorem: shah}
The period map \eqref{equation: period of sextic curve} defined above is an algebraic open embedding with image the complement of $\widehat{\Gamma}\bs (\calH_{\Delta} \cup \calH_{\infty})$. Moreover $\Prd$ extends to an isomorphism $\calM_1\cong \widehat{\Gamma}\bs (\widehat{\DD}-\calH_{\infty})$, and further to $\overline{\calM}\cong \overline{\widehat{\Gamma}\bs \widehat{\DD}}^{\calH_{\infty}}$.
\end{thm}

\begin{rmk}
\label{remark: description of looi compactification}
The Looijenga compactification $\overline{\widehat{\Gamma}\bs\widehat{\DD}}^{\calH_{\infty}}$ can be constructed from the Baily-Borel compactification $\overline{\widehat{\Gamma}\bs\widehat{\DD}}^{bb}$ in the following way. First we blow up some boundary components of $\overline{\widehat{\Gamma}\bs\widehat{\DD}}^{bb}$ to make the closure of each hypersurface in $\widehat{\Gamma}\bs\calH_{\infty}$ Cartier. Then we contract the strict transform of $\overline{\widehat{\Gamma}\bs\calH_{\infty}}$ to a point. The corresponding point in $\overline{\calM}$ is a semi-stable sextic, represented by the cube of a quadratic polynomial.
\end{rmk}

\section{Moduli of Nodal Sextic Curves}
\label{section: period map}
In this section, we will first characterize the moduli spaces of nodal sextic curves via $\GIT$ constructions (\S \ref{subsection: GIT of singular sextic curves}). Using the Hodge structures of certain $K3$ surfaces naturally associated with nodal sextic curves, we can realize those moduli spaces as arithmetic quotients of type $\IV$ domains (\S \ref{subsection: period map}).

Consider the subspace $\Sigma^n$ of $\PP\Sym^6(V^*)$ consisting of plane sextic curves with $n$ different nodes. A singular type $T$ is defined to be an irreducible component of $\Sigma^n$. We also denote by $\Sigma_T$ to be the irreducible component corresponding to the type $T$. A plane sextic curve in $\Sigma_T$ is called to have singular type $T$. In the rest of this section, we fix $T$ and write $\Sigma=\Sigma_T$ for short.

\subsection{GIT construction of nodal sextic curves}
\label{subsection: GIT of singular sextic curves}
Let $\overline{\Sigma}$ be the Zariski closure of $\Sigma$ and let $f\colon\widetilde{\Sigma}\to \overline{\Sigma}$ be the normalization of $\overline{\Sigma}$. Denote by $\calO(1)$ the natural polarization on $\overline{\Sigma}$ induced from the ambient space $\PP\Sym^6(V^*)$. The group $\SL(V)$ acts equivariantly on the polarized varieties 
\begin{equation}
\label{equation: morphisms}
(\widetilde{\Sigma}, f^*\calO(1))\to (\overline{\Sigma}, \calO(1))\hookrightarrow (\PP\Sym^6(V^*),\calO(1)).
\end{equation} 
Denote by $\overline{\Sigma}^{ss}, \widetilde{\Sigma}^{ss}$ and $\PP\Sym^6(V^*)^{ss}$ the spaces of semi-stable points with respect to the actions of $\SL(V)$. By \cite{severi1968curve}, $\Sigma$ is smooth, hence can be identified with its preimage in $\widetilde{\Sigma}$. By \cite{shah1980completek3}, the points in $\Sigma$ is stable under the action of $\SL(V)$. 

\begin{de}
We define $\F=\SL(V)\dbs \Sigma$ to be the moduli space of singular sextic curves of type $T$. Define $\overline{\F}=\SL(V)\dbs \widetilde{\Sigma}^{ss}$, which is a natural compactification of $\F$.
\end{de}
By \cite[theorem 1.19]{mumford1994geometric}, $\widetilde{\Sigma}^{ss}$ is the preimage of $\overline{\Sigma}^{ss}$ in $\widetilde{\Sigma}$. Taking quotient of \eqref{equation: morphisms} by $\SL(V)$, we obtain morphisms among GIT quotients:
\begin{equation*}
j\colon \overline{\F}=\SL(V)\dbs \widetilde{\Sigma}^{ss} \to \SL(V)\dbs\overline{\Sigma}^{ss} \hookrightarrow \overline{\calM}=\SL(V)\dbs \PP\Sym^6(V^*)^{ss}
\end{equation*}
with the first map the normalization of $\SL(V)\dbs\overline{\Sigma}^{ss}$. 
In conclusion, we have
\begin{prop}
\label{proposition: j}
The morphism $j\colon \overline{\F}\to \overline{\calM}$ is a normalization of its image $j(\overline{\F})$.
\end{prop}

\subsection{Period domain and period map}
\label{subsection: period map}
Let $F\in \Sym^6(V^*)$ with $Z(F)\in \Sigma$. We denote by $S_F$ the double cover of $\PP(V)$ branched along $Z(F)$. Then $S_F$ is a $K3$ surface with $n$ nodal singularities on the preimage of $Z(F)$. Let $Q_F$ be the blowup of $\PP(V)$ at the $n$ nodes of $Z(F)$. Assume $Z(F)$ has $l$ irreducible components. The proper transforms of those components in $Q_F$ are $l$ disjoint curves, denoted by $C_1, C_2, \cdots, C_l$. Let $\widetilde{S}_F$ be the double cover of $Q_F$ branched along $\bigcup_{i=1}^l C_i$. Then $\widetilde{S}_F$ is a smooth $K3$ surface which resolves the nodal sigularities of $S_F$. There is an anti-symplectic involution $\iota$ on $\widetilde{S}_F$ induced by the double-cover construction. We use the same notation to denote the induced involution of $H^2(\widetilde{S}_F)$. Define
\begin{equation*}
M_F \coloneqq \{x\in H^2(\widetilde{S}_F,\ZZ)|\iota x=x\}
\end{equation*}
which is a primitive sublattice of $H^2(\widetilde{S}_F,\ZZ)$. Since the action of $\iota$ on $\widetilde{S}_F$ is anti-symplectic, the lattice $M_F$ has signature $(1, n^{\prime})$. Here $n^{\prime}$ is a non-negative integer and we will show $n^{\prime}=n$ from the next lemma.

\begin{lem}
\label{lemma: invariant cohomology}
Let $X\longrightarrow Y$ be a branched cyclic cover between two closed manifolds with branch locus $B$ a closed submanifold of $Y$ with codimension $2$. Suppose $\iota$ is a generator of the Deck transformation group, and denote by $H^*(X,\QQ)^{\iota}$ the invariant subspace of $H^*(X,\QQ)$ under the induced action of $\iota$. Then we have an isomorphism $H^*(Y, \QQ)\cong H^*(X,\QQ)^{\iota}$.
\end{lem}
\begin{proof}
The preimage of $B$ in $X$ is also denoted by $B$. Let $N_Y(B)$ be a tubular neighbourhood of $B$ in $Y$, and $S_Y(B)$ be the boundary of $N_Y(B)$ which is a circle bundle over $B$. Take $Y^{\circ}\coloneqq Y-N_Y(B)$. Let $N_X(B), S_X(B)$ and $X^{\circ}$ be the preimage of $N_Y(B), S_Y(B)$ and $Y^{\circ}$ in $X$, respectively. The Mayer-Vietoris exact sequence gives the following commutative diagram:
\begin{equation}
\label{diagram: five lemma}
\begin{tikzcd}
\cdots \arrow{r} & H^k(Y) \arrow{d}\arrow{r} & H^k(N_Y(B))\oplus H^k(Y^{\circ})\arrow{d}\arrow{r} & H^k(S_Y(B)) \arrow{d} \arrow{r} & \cdots \\
\cdots \arrow{r} & H^k(X)\arrow{r} & H^k(N_X(B))\oplus H^k(X^{\circ})\arrow{r}  & H^k(S_X(B)) \arrow{r} & \cdots
\end{tikzcd}
\end{equation}
where all the cohomology groups have $\QQ$-coefficients. The second long exact sequence in diagram \eqref{diagram: five lemma} admits an induced action by $\iota$. The image of the first exact sequence in the second is contained in the $\iota$-invariant part. Since $N_X(B)\longrightarrow N_Y(B)$ is a homotopy equivalence, the morphism $H^k(N_Y(B))\longrightarrow H^k(N_X(B))$ is an isomorphism. The maps $X^{\circ}\longrightarrow Y^{\circ}$ and $S_X(B)\longrightarrow S_Y(B)$ are regular covers, hence the maps $H^k(Y^{\circ})\longrightarrow H^k(X^{\circ})^{\iota}$ and $H^k(S_Y(B))\longrightarrow H^k(S_X(B))^{\iota}$ are isomorphisms. Therefore we have isomorphism $H^k(Y,\QQ)\cong H^k(X,\QQ)^{\iota}$ by five lemma.
\end{proof}

Let $H$ be the total transform of the hyperplane class in $Q_F$ and $E_i$ ($1\leq i\leq n$) be the exceptional divisors in $Q_F$. These $n+1$ cycles form an integral basis for the lattice $H^2(Q_F,\ZZ)$. By lemma \ref{lemma: invariant cohomology}, the rank of $M$ is equal to $n+1$. Consider the embedding of lattice $H^2(Q_F, \ZZ)(2)$\footnote{For a lattice $L$, we denote by $L(n)$ the lattice with the same underlying abelian group while the intersection pairing rescaled by $n$.} into $H^2(\widetilde{S}_F,\ZZ)$, which has image contained in $M$. We use $[H]$ and $[E_i]$ to denote the corresponding cohomology classes in $H^2(\widetilde{S}_F,\ZZ)$. The preimage of $C_i$ in $\widetilde{S}_F$ is still denoted by $C_i$. The proper transform of $\calO_{Q_F}(C_i)$ to $\widetilde{S}_F$ is $\calO_{\widetilde{S}_F}(2C_i)$. We denote by $[C_i]$ the first Chern class of $\calO_{\widetilde{S}_F}(C_i)$. Let $d_i$ be the degree of the image of $C_i$ in $\PP(V)$. Then
\begin{equation}
\label{equation: expression}
[C_i]={\frac{d_i[H]-\sum_{j=1}^n a_{ij}[E_j]}{2}}.
\end{equation}
Here $a_{ij}=1$ if $E_j$ is from blowing up of an intersection point between image of $C_i$ in $\PP(V)$ and another component of $Z(F)$, $a_{ij}=2$ if $E_j$ is from blowing up of a node on the image of $C_i$ in $\PP(V)$ and $a_{ij}=0$ otherwise. We have the following description of $M_F$.

\begin{prop}
\label{proposition: m}
The primitive sublattice $M_F$ is generated by $H^2(Q_F, \ZZ)(2)$ and $[C_1]$, $[C_2]$, $\cdots$, $[C_l]$. Moreover
\begin{equation*}
\frac{M_F}{H^2(Q_F, \ZZ)(2)}\cong (\ZZ/2)^{l-1}.
\end{equation*}
\end{prop}

\begin{proof}
Take $M_1=H^2(Q_F, \ZZ)(2)$ and $\widetilde{M}$ the sublattice generated by $M_1$ and $[C_1], [C_2], \cdots, [C_l]$. We have $\widetilde{M}\subset M_F$ with the same rank. Since $H^2(Q_F, \ZZ)$ is a unimodular lattice, we have $M_1^*/M_1\cong (\ZZ/2)^{n+1}$. Since $M_1\subset\widetilde{M}\subset M_F\subset M_1^*$, the quotients $M_F/M_1$ and $\widetilde{M}/M_1$ are both $\ZZ/2$-vector spaces. Tensoring the short exact sequence
\begin{equation*}
0\to H^2(Q_F, \ZZ)\to H^2(\widetilde{S}_F, \ZZ)\to H^2(\widetilde{S}_F, \ZZ)/M_1\to 0
\end{equation*}
with $\ZZ/2$, we obtain a long exact sequence
\begin{equation}
\label{sequence: tor}
0\to \Tor(H^2(\widetilde{S}_F, \ZZ)/M_1, \ZZ/2) \to H^2(Q_F, \ZZ/2)\to H^2(\widetilde{S}_F, \ZZ/2)\to (H^2(\widetilde{S}_F, \ZZ)/M_1)\otimes \ZZ/2\to 0.
\end{equation}
So the $\ZZ/2$-rank of the kernel of $H^2(Q_F, \ZZ/2)\to H^2(\widetilde{S}_F, \ZZ/2)$ is equal to the $\ZZ/2$-rank of the torsion part $(H^2(\widetilde{S}_F, \ZZ)/M_1)_{\text{tor}}$ of $H^2(\widetilde{S}_F,\ZZ)/M_1$. On the other hand, according to theorem 1 in \cite{lee1995doublebranchedcover}, we have the exact sequence
\begin{equation}
\label{sequence: lee}
0\to H^1(Q_F, \bigcup_{i=1}^l C_i, \ZZ/2)\to H^2(Q_F, \ZZ/2)\to H^2(\widetilde{S}_F, \ZZ/2).
\end{equation}
The long exact sequence for relative cohomology groups
\begin{equation}
\label{sequence: relative}
0\to H^0(Q_F,\ZZ/2)\to H^0(\bigcup_{i=1}^l C_i, \ZZ/2)\to H^1(Q_F, \bigcup_{i=1}^l C_i, \ZZ/2)\to 0
\end{equation}
implies that the rank of $H^1(Q_F, \bigcup_{i=1}^l C_i, \ZZ/2)$ is $l-1$. Combining sequences \eqref{sequence: tor}, \eqref{sequence: lee} and \eqref{sequence: relative}, we conclude that the $\ZZ/2$-rank of $(H^2(\widetilde{S}_F, \ZZ)/M_1)_{\text{tor}}$ is equal to $l-1$. Since $M_F$ is the saturation of $M_1$ in $H^2(\widetilde{S}_F, \ZZ)$, we have $M_F/M_1\cong (\ZZ/2)^{l-1}$ as abelian groups.

We claim that $[C_1], [C_2], \dots, [C_{l-1}]\in \widetilde{M}/M_1$ are $\ZZ/2$-independent.

Apparently $l\le 6$. For $l=1$, the claim is clear. Suppose $l\ge 2$. For any $i\in \{1,2,\dots, l-1\}$, take $s(i)\in\{1,2,\dots,n\}$, such that $E_{s(i)}$ is from blowing up of an intersection point of images of $C_i$ and $C_l$ in $\PP(V)$. Then $a_{is(i)}=1$. Thus the coefficient of $[E_{s(i)}]$ in the expression \eqref{equation: expression} of $[C_i]$ is equal to $-\frac{1}{2}$. Therefore, any nontrivial $\ZZ/2$-linear combination of $[C_1], [C_2],\dots,[C_{l-1}]$ can not vanish. The claim follows.

Then the $\ZZ/2$-rank of $\widetilde{M}/M_1$ is at least $l-1$. Since $\widetilde{M}/M_1\subset M_F/M_1 \cong (\ZZ /2)^{l-1}$, we have $\widetilde{M}/M_1\cong (\ZZ/2)^{l-1}$ and hence $M_F=\widetilde{M}$.
\end{proof}

The isomorphism type of $(H^2(\widetilde{S}_F, \ZZ), M_F, H, \iota)$ does not depend on the choice of $F\in \Sigma$. Let $(\Lambda, M, H, \iota)$ be an abstract data isomorphic to $(H^2(\widetilde{S}_F, \ZZ), M_F, H, \iota)$. Let $\Lambda_T$ be the orthogonal complement of $M$ in $\Lambda$. We use the same notation for the corresponding elements of $[H],[E_i]$ in $M$. We have a family $p\colon \calS\to \Sigma$ of smooth $K3$ surfaces with anti-symplectic involution $\iota$. The Hodge structures of $K3$ surfaces give rise to a variation of polarized Hodge structures on the local system $R^2p_*(\ZZ)$ with induced involution $\iota$. Let $\HH$ be $(-1)$-eigensubsheaf of $R^2p_*(\ZZ)$. We have a subvariation of Hodge structures on $\HH$, which is naturally polarized by the topological intersection pairing $\varphi$.

\begin{de}
We define the period domain $\DD=\DD_T$ of type $T$ to be the connected component of
\begin{equation*}
\PP\{x\in \Lambda_T\otimes \CC|\varphi(x, x)=0, \varphi(x,\overline{x})>0\}
\end{equation*}
contained in $\widehat{\DD}$.
The arithmetic group $\Gamma$ is defined to be the centralizer of $\iota$ in $\widehat{\Gamma}$. Equivalently, $\Gamma=\{\sigma\in\widehat{\Gamma}\big{|}\sigma(M)=M\}$ can be viewed as the group of automorphisms of the data $(\Lambda, M, H, \iota)$ with spinor norm $1$.
\end{de}

The variation of Hodge structures $\HH$ induces a period map $\Prd \colon \Sigma\longrightarrow \Gamma\bs \DD$. Since the group $\SL(V)$ acts equivariantly on $\HH$, the period map is constant on each $\SL(V)$-orbit in $\Sigma$. Therefore, the map $\Prd$ descends to $\F_T\longrightarrow \Gamma\bs \DD$, still denoted by $\Prd$.

Define $\calH_*$ to be the intersection of $\calH_{\infty}$ with $\DD$, which is a $\Gamma$-invariant hyperplane arrangement. We have the Looijenga compactification $\overline{\Gamma\bs\DD}^{\calH_*}$ of $\Gamma\bs(\DD-\calH_*)$.

\begin{proof}[Proof of theorem \ref{theorem: main}]

There is a natural finite morphism $\pi\colon \Gamma\bs \DD\longrightarrow \widehat{\Gamma}\bs \widehat{\DD}$ (see \cite[Proposition A.1]{yu2018moduli} for finiteness). We have a  commutative diagram
\begin{equation*}
\label{diagram: normalization}
\begin{tikzcd}
\F\arrow{r}{\Prd}\arrow[hook]{d}{j}  & \Gamma\bs\DD\arrow{d}{\pi} \\
\calM \arrow[hook]{r}{\Prd} & \widehat{\Gamma}\bs\widehat{\DD}.
\end{tikzcd}
\end{equation*}
Since the composed map from $\F$ to $\widehat{\Gamma}\bs\widehat{\DD}$ is injective, the map $\Prd\colon\F\to \Gamma\bs \DD$ is also injective. The image of $\F$ in $\widehat{\Gamma}\bs\widehat{\DD}$ does not meet $\widehat{\Gamma}\bs \calH_{\infty}$, thus the image of $\F$ in $\Gamma\bs\DD$ is contained in $\Gamma\bs (\DD-\calH_{*})$. 

It is a classical result (\cite{severi1959nodalcurve}, also see \cite[Chapter 4.7]{sernesi2006deformation}) that $\dim \Sigma=27-n$, hence $\dim \F=\dim \Sigma-\dim \PSL(V)=19-n$. On the other hand, $\dim \DD=\rank (\Lambda_T)-2=20-\rank(M)=19-n$. Thus $\F$ and $\Gamma\bs \DD$ are irreducible quasi-projective varieties with the same dimension. We conclude that $\Prd$ is a bimeromorphism.

We claim that $\pi$ is generically injective. The argument here follows \cite[Proposition A.1]{yu2018moduli}. Recall that $\Gamma$ is the centralizer of $\iota$ in $\widehat{\Gamma}$. For any point $x\in \DD$, denote by $\Pic(x)=H_x^{1,1}\cap \Lambda$ the Picard lattice. Take two points $x,y\in \DD$ with $\Pic(x)=M$. Suppose there exists $\sigma\in \widehat{\Gamma}$ with $\sigma(x)=y$. Then $\sigma^{-1}(\Pic(y))\subset \Pic(x)=M$, hence $\Pic(y)=M$ and $\sigma$ preserves $M$. Thus $\sigma\in \Gamma$. This implies the generic injectivity of $\pi$. By \cite[Theorem A.13]{yu2018moduli}, the morphism $\pi$ extends to a finite morphism $\overline{\Gamma\bs\DD}^{\calH_*}\longrightarrow \overline{\widehat{\Gamma}\bs\widehat{\DD}}^{\calH_{\infty}}$, still denoted by $\pi$. We conclude that $\pi\colon \overline{\Gamma\bs\DD}^{\calH_*}\longrightarrow\overline{\widehat{\Gamma}\bs\widehat{\DD}}^{\calH_{\infty}}$ is a normalization of its image.

We now have the following commutative diagram:
\begin{equation}
\label{diagram: normalization}
\begin{tikzcd}
\F\arrow{r}{\Prd}\arrow{d}{j}      & \Gamma\bs(\DD-\calH_{*})\arrow{d}{\pi} \\
\overline{\calM}\arrow{r}{\Prd} & \overline{\widehat{\Gamma}\bs\widehat{\DD}}^{\calH_{\infty}}.
\end{tikzcd}
\end{equation}
Since $\F$ and $\Gamma\bs\DD$ are bimeromorphic via $\Prd$, the Zariski closure of $j(\F)$ in $\overline{\calM}$ and that of $\pi(\Gamma\bs\DD)$ in $\overline{\widehat{\Gamma}\bs\widehat{\DD}}^{\calH_{\infty}}$ are identified via $\Prd$. By Proposition \ref{proposition: j}, $\overline{\F}$ is the normalization of the closure of $j(\F)$ in $\overline{\calM}$. Therefore, by taking normalizations of the two closures, we obtain a unique isomorphism between $\overline{\F}$ and $\overline{\Gamma\bs\DD}^{\calH_*}$, which extends the upper row morphism in diagram \eqref{diagram: normalization}, and is still denoted by $\Prd$. This argument essentially follows from a slightly more general result \cite[Lemma 5.2]{yu2018moduli}. The theorem follows.
\end{proof}

We provide another proof of the injectivity of $\Prd\colon \F \to \Gamma\bs \DD$ which does not rely on the injectivity of $\Prd\colon \calM\to \widehat{\Gamma}\bs \widehat{\DD}$. We call an element in a lattice with self-intersection $-2$ a root. Let $H_M^{\perp}$ be the orthogonal complement of $H$ in $M$. We need the following lemma (which is a property of lattices with independent interest):
\begin{lem}
\label{lemma: no extra roots}
The vectors $\pm E_i\in M$ are all the roots in $H_M^{\perp}$.
\end{lem}
\begin{proof}
Suppose not, let $v$ be a root in $H_M^{\perp}$ which is distinct from $\pm E_1, \pm E_2, \dots, \pm E_n$. From proposition \ref{proposition: m}, $M$ is integrally generated by $[H], [E_1], [E_2], \dots, [E_n], [C_1], [C_2], \dots, [C_l]$. We can write $v$ as an integral linear combination of $[H], [E_1], [E_2], \dots, [E_n], [C_1], [C_2], \dots, [C_l]$. Suppose the coefficient of $[C_i]$ is $\epsilon_i$. Let $C^{\prime}$ be the union of images of those $C_i$ in $\PP(V)$ with $\epsilon_i$ odd, and $C^{\prime\prime}$ be the union of images of other $C_i$ in $\PP(V)$.

There is a unique expression of $v$ as a half-integral combination of $[E_1], [E_2], \dots, [E_n]$. In this expression, the coefficient of $[E_j]$ is congruent to $\frac{1}{2}$ (mod $\ZZ$) if and only if $E_j$ is from blowing up of certain intersection point of $C^{\prime}$ and $C^{\prime\prime}$. By the assumption on $v$, there is at least one $j\in\{1,2,\dots,n\}$ with the coefficient of $[E_j]$ congruent to $\frac{1}{2}$, hence both $C^{\prime}$ and $C^{\prime\prime}$ are nonempty. Since $\deg(C^{\prime})+\deg(C^{\prime\prime})=6$, there are at least $5$ intersection points between $C^{\prime}$ with $C^{\prime\prime}$. Choose $5$ intersection points, and let the corresponding exceptional rational curves be $E_{j_1},E_{j_2},E_{j_3},E_{j_4},E_{j_5}$. Notice that the orthogonal complement of $[H]$ in $M$ is negative definite, we have
\begin{equation*}
-2= v^2\le (\frac{[E_{j_1}]+[E_{j_2}]+[E_{j_3}]+[E_{j_4}]+[E_{j_5}]}{2})^2=-\frac{5}{2},
\end{equation*}
contradiction! The lemma follows.
\end{proof}

Now we give the second proof of the injectiviey of $\Prd\colon \F \to \Gamma\bs \DD$. Suppose $F_1,F_2\in \Sigma$ satisfy $\Prd(F_1)=\Prd(F_2)$. Then there exists an isomorphism $\kappa\colon (H^2(\widetilde{S}_{F_1}), M_{F_1}, H, \iota)\cong(H^2(\widetilde{S}_{F_2}), M_{F_2}, H, \iota)$, sending $H^{2,0}(\widetilde{S}_{F_1})$ to $H^{2,0}(\widetilde{S}_{F_2})$. Notice that $\kappa$ maps roots in $H_{M_{F_1}}^{\perp}$ to roots in $H_{M_{F_2}}^{\perp}$. By lemma \ref{lemma: no extra roots}, the effective roots $[E_i]\in M_{F_1}$ are sent to $\pm[E_j]$ in $M_{F_2}$. Denote by $r_i$ the reflection of $H^2(\widetilde{S}_F,\ZZ)$ with respect to $[E_i]$. We can compose $\kappa$ with some reflections $r_j$ such that effective roots are sent to effective roots. Suppose $\kappa$ sends $[E_i]$ to $-[E_j]$, then we use $r_j\circ \kappa$ instead of $\kappa$. After adjustments, we may assume $\kappa$ sends exceptional divisor classes $[E_1], \cdots, [E_n]$ in $H_{M_{F_1}}^{\perp}$ to those in $H_{M_{F_2}}^{\perp}$. Furthermore, the class $16H-[E_1]-[E_2]-\cdots-[E_n]$ is ample on both $\widetilde{S}_{F_1}$ and $\widetilde{S}_{F_2}$. The map $\kappa$ preserves this ample class, hence (by global Torelli theorem) is induced by an isomorphism $f\colon \widetilde{S}_{F_2}\cong \widetilde{S}_{F_1}$, which commutes with $\iota$. So $f$ descends to an isomorphism $Q_{F_2}\cong Q_{F_1}$. This morphism preserves $\{E_1, \cdots, E_n\}$ on both sides. So it descends to a linear transformation $\PP(V)\to \PP(V)$ sending $Z(F_2)$ to $Z(F_1)$. The injectivity follows.

When $\calH_*$ is empty, the Looijenga compactification $\overline{\Gamma\bs\DD}^{\calH_*}$ is actually the Baily-Borel compactification $\overline{\Gamma\bs\DD}^{bb}$ of $\Gamma\bs\DD$. In the remaining of this section, we give a criterion to determine whether $\overline{\F}$ is identified with the Baily-Borel compactification $\overline{\Gamma\bs\DD}^{bb}$. 
\begin{prop}
\label{proposition: bb criterion}
Let $T$ be a singular type. The hyperplane arrangement $\calH_*$ is empty if and only if $\Sigma_T$ does not contain any cube of quadric polynomial.
\end{prop}
\begin{proof}
This is straightforward from remark \ref{remark: description of looi compactification} and diagram \eqref{diagram: normalization}.
\end{proof}

\section{Examples and Related Constructions}
\label{section: example}

\subsection{Relation to del Pezzo surfaces}
\label{section: del pezzo}
Let $T$ be a singular type of nodal sextic curves. Suppose a generic sextic curve $Z(F)$ of type $T$ is not irreducible, then it contains a line, a quadric curve or a cubic curve. Suppose $Z(F)$ contains a line, then there are at least $5$ nodes on that line. Suppose $Z(F)$ contains a quadric curve, then there are at least $8$ nodes on that quadric curve. Suppose $Z(F)$ contains a cubic curve, then there are at least $9$ nodes on that cubic curve. In any cases, the nodes are not of general position.

Now we assume that a generic sextic curve of type $T$ is irreducible. Let $n$ be the number of nodes. Since a sextic curve has arithmetic genus $10$, we have $n\le 10$. It is well-known that for generically $n\le 8$ points on $\PP^2$, there exists an irreducible plane sextic curve containing those points as nodes. This implies the following relation between nodal sextic curves and del Pezzo surfaces. For reader's convenience, we include a proof.

\begin{prop}
Suppose a generic curve $Z(F)$ of singular type $T$ is an irreducible nodal sextic curve with $n\le 8$ nodes, then the surface $Q_F$ (which is the blowup of $\PP(V)$ at the $n$ nodes of $Z(F)$) is a del Pezzo surface of degree $9-n$, and the double cover $\widetilde{S}_F\longrightarrow Q_F$ is branched along a smooth irreducible curve of genus $10-n$.
\end{prop}
\begin{proof}
Let $S_n$ be the permutation group of $n$ elements acting naturally on $(\PP^2)^n$.
Consider the map
\begin{equation*}
f\colon \Sigma_T\longrightarrow S_n\bs ((\PP^2)^n-\bigcup_{i\ne j}\Delta_{ij})
\end{equation*}
sending $C\in \Sigma_T$ to the set of its nodes. Here $\Delta_{ij}$ consists of points in $(\PP^2)^n$ with the $i$-th and $j$-th coordinates equal. We claim that the tangent map of $f$ is surjective at any $C\in \Sigma_T$. Let $p_1,\dots, p_n$ be the $n$ nodes of $C$. Let $\widetilde{C}$ be the normalization of $C$, and denote by $q_{2i-1}, q_{2i}$ the preimage of $p_i$. Let $D=p_1+\cdots +p_n$ be a divisor on $C$, and $\widetilde{D}=q_1+\cdots q_{2n}$ be a divisor on $\widetilde{C}$. By local calculation of variation of nodes, the tangent space of $\Sigma_T$ at $C$ is naturally identified with:
\begin{equation*}
H^0(C,\calO(6)(-D))\cong H^0(\widetilde{C},\pi^*\calO(6)(-\widetilde{D})).
\end{equation*}
The tangent map $df$ at $C$ is:
\begin{equation*}
df\big{|}_C\colon H^0(\widetilde{C},\pi^* \calO(6)(-\widetilde{D}))\longrightarrow\bigoplus_i \pi^*(\calO(6))(-\widetilde{D})\big{|}_{q_i}
\end{equation*}
sending $s$ to $(s(q_i))_i$. We claim that $df\big{|}_C$ is surjective. From the following exact sequence of coherent sheaves on $\widetilde{C}$:
\begin{equation*}
0\longrightarrow \pi^*\calO(6)(-2\widetilde{D})\longrightarrow \pi^*\calO(6)(-\widetilde{D})\longrightarrow \bigoplus_i \pi^*\calO(6)(-\widetilde{D})\big{|}_{q_i}\longrightarrow 0
\end{equation*}
it suffices to show $H^1(\widetilde{C},\pi^*\calO(6)(-2\widetilde{D}))=0$. The degree of $\pi^*\calO(6)(-2\widetilde{D})$ equals to $36-4n$. The genus of $\widetilde{C}$ is $g(\widetilde{C})=10-n$. Since $n\le 8$, we have $36-4n> 2g(\widetilde{C})-2$. Thus $H^1(\widetilde{C},\pi^*\calO(6)(-2\widetilde{D}))=0$ and the claim follows.

It is now clear that a generic point $C=Z(F)\in \Sigma_T$ has $n$ nodes of general position. Therefore, the surface $Q_F$ is a del Pezzo surface of degree $9-n$. Moreover, the double cover $\widetilde{S}_F\longrightarrow Q_F$ is branched along $\widetilde{C}$, which is a smooth irreducible curve of genus $10-n$.
\end{proof}

\begin{rmk}
A generic choice of $9$ points on $\PP^2$ can not be realized nodes of any sextic curve. See \cite{cayley1869/71sextic}.
\end{rmk}

\subsection{Union of smooth plane curves}
\label{subsection: union}
In this section, we apply theorem \ref{theorem: main} to describe the moduli spaces of tuples of smooth plane curves with total degree $6$. Special cases were studied in \cite{laza2009deformation}, \cite{matsumoto1992sixlines}, \cite{gallardo2017compactifications}. Let $T$ be the singular type such that a generic sextic curve in $\Sigma_T$ is the union of several smooth curves. Suppose the degrees of those smooth components are $m_1\le m_2\le \cdots \le m_l$, then we also denote $T=(m_1,\dots, m_l)$. We have $11$ possibilities of such $T$. They are $(6)$, $(1,5)$, $(2,4)$, $(3,3)$, $(1,1,4)$, $(1,2,3)$, $(2,2,2)$, $(1,1,1,3)$, $(1,1,2,2)$, $(1,1,1,1,2)$ and $(1,1,1,1,1,1)$.

Now we give an explicit description of the GIT quotient $\F_T$ for the types above. Let $n_1\cdot 1+ n_2\cdot 2+ \cdots  +n_k\cdot k=6$ be a partition of $6$. Then we have a natural finite map
\begin{equation*}
f\colon \prod_i \PP(\Sym^i V^*)^{n_i}\longrightarrow \PP(\Sym^6 V^*).
\end{equation*}
Taking quotient by the action of permutation group $\prod_i S_{n_i}$ on the left, we have a generically injective finite map
\begin{equation*}
\prod_i S_{n_i}\bs \prod_i \PP(\Sym^i V^*)^{n_i}\longrightarrow \PP(\Sym^6 V^*).
\end{equation*}
which is the normalization of $\overline{\Sigma}$. So we have
\begin{equation*}
\widetilde{\Sigma}\cong \prod_i S_{n_i}\bs \prod_i \PP(\Sym^i V^*)^{n_i}.
\end{equation*}

Next we describe the polarization on $\widetilde{\Sigma}$. The pullback of $\calO(1)$ under $f$ is
\begin{equation*}
f^*(\calO(1))\cong \calO(1)^{\boxtimes(n_1+\cdots+n_k)}.
\end{equation*}
So the GIT construction has the following form
\begin{equation}
\label{equation: GITspaces}
(\SL(V)\times \prod_i S_{n_i})\dbs (\prod_i \PP(\Sym^i V^*)^{n_i}, \calO(1)^{\boxtimes(n_1+\cdots+n_k)}).
\end{equation}

In other words, the GIT quotients of the form \eqref{equation: GITspaces} are isomorphic to Looijenga compactifications of arithmetic quotients of type $\IV$ domains. From proposition \ref{proposition: bb criterion}, we have:
\begin{prop}
\label{proposition: bb smooth components}
For the following $T$: $(6)$, $(1,5)$, $(3,3)$, $(1,1,4)$, $(1,2,3)$, $(1,1,1,3)$, $(1,1,2,2)$, $(1,1,1,1,2)$, $(1,1,1,1,1,1)$, we obatin identifications of $\overline{\F}_T$ with the Baily-Borel compactifications $\overline{\Gamma\bs\DD}^{bb}$.
\end{prop}

We put together the information in table \eqref{table: smooth components}.
\begin{table}[ht] \caption{Information of singular types $T$ when a generic member in $\Sigma_T$ is the union of some smooth curves}
\vspace{0.5ex}
\label{table: smooth components}
\centering
\begin{tabular}{cccccc}
\hline\hline
Type & Number of Nodes & $Dim(\F_T)$ &   $Rank(M)$ & $A_M$ & Whether Baily-Borel \\ [0.5ex]
\hline
$(6)$           & $0$               & $19$        & 1   &  $(\ZZ/2\ZZ)^1$                       & no        \\
$(1,5)$         & $5$               & $14$        & 6   &  $(\ZZ/2\ZZ)^4$                       & yes        \\
$(2,4)$         & $8$               & $11$        & 9   &  $(\ZZ/2\ZZ)^7$                       & no         \\
$(3,3)$         & $9$               & $10$        & 10  &  $(\ZZ/2\ZZ)^8$                       & yes        \\
$(1,1,4)$       & $9$               & $10$        & 10  &  $(\ZZ/2\ZZ)^6$                       & yes        \\
$(1,2,3)$       & $11$              & $8$         & 12  &  $(\ZZ/2\ZZ)^8$                       & yes        \\
$(2,2,2)$       & $12$              & $7$         & 13  &  $(\ZZ/2\ZZ)^9$                       & no         \\
$(1,1,1,3)$     & $12$              & $7$         & 13  &  $(\ZZ/2\ZZ)^7$                       & yes        \\
$(1,1,2,2)$     & $13$              & $6$         & 14  &  $(\ZZ/2\ZZ)^8$                       & yes        \\
$(1,1,1,1,2)$   & $14$              & $5$         & 15  &  $(\ZZ/2\ZZ)^7$                       & yes        \\
$(1,1,1,1,1,1)$ & $15$              & $4$         & 16  &  $(\ZZ/2\ZZ)^6$                       & yes        \\
\hline
\end{tabular}
\end{table}
\section{Moduli of Nodal and Symmetric Sextic Curves}
\label{section: with symmetry}
We follow \cite{yu2018moduli} to define the so-called symmetry type. Let $A$ be a finite subgroup of $\SL(V)$ which contains the center $\mu_3\subset \SL(V)$. For any $\zeta\in\mu_3$ and $F\in \Sym^6(V^*)$, we have $\zeta(F)=F\circ \zeta^{-1}=\zeta^{-6}F=F$. Let $\lambda$ be a character of $A$ sending $\zeta$ to $1$. Let $\calV_{\lambda}$ be the eigenspace of $\Sym^6(V^*)$ associated with $\lambda$. Two pairs $(A_1,\lambda_1)$ and $(A_2,\lambda_2)$ are called equivalent if and only if there exists $g\in \SL(V)$ such that $g A_1 g^{-1}=A_2$ and $\lambda_2(gag^{-1})=\lambda_1(a)$ for all $a\in A_1$. In this case we write $(A_2,\lambda_2)=g(A_1,\lambda_1)$.
\begin{de}
A symmetry type for sextic curves is an equivalence class of pairs $(A,\lambda)$, denoted by $[(A,\lambda)]$.
\end{de}

Next we consider moduli of sextic curves with singular type $T_{sin}$ and symmetry type $T_{sym}$. Recall that $\Sigma_{T_{sin}}$ is the irreducible space of plane sextic curves of singular type $T_{sin}$. Let $n$ be the number of nodes. Choose a representative $(A,\lambda)$ for the symmetry type $T_{sym}$. Define $\Sigma_{T_{sin},T_{sym}}\coloneqq \Sigma_{T_{sin}}\cap \PP\calV_{\lambda}$. This may be reducible, see example \ref{example: six involution}. We denote by $\Sigma$ one of the irreducible components of $\Sigma_{T_{sin},T_{sym}}$. Let $\overline{\Sigma}$ be the closure of $\Sigma$ in $\PP \Sym^6(V^*)$, and $\widetilde{\Sigma}$ the normalization of $\overline{\Sigma}$. Define
\begin{equation*}
N\coloneqq\{g\in\SL(V)\big{|}gAg^{-1}=A, \lambda(a)=\lambda(gag^{-1}),\forall a\in A, g\in \SL(V)\}
\end{equation*}
which is a reductive subgroup of $\SL(V)$. For any $g\in N$, $x\in \Sigma$ and $a\in A$, we have
\begin{equation*}
a(gx)=g(g^{-1}ag)(x)=g\lambda(g^{-1}ag)x=\lambda(a)gx,
\end{equation*}
which implies that $gx\in \Sigma$. Thus $N$ naturally acts on $\Sigma$, hence also on $\overline{\Sigma}$ and $\widetilde{\Sigma}$. The points in $\Sigma$ are stable under the action of $N$. Define $\F\coloneqq N\dbs\Sigma$ and $\overline{\F}\coloneqq N\dbs\widetilde{\Sigma}^{ss}$. From \cite{luna1975Adh} and the argument in the proof of proposition 2.7 in \cite{yu2018moduli}, we have:
\begin{prop}
\label{proposition: with symmetry finite morphism}
There is a natural morphism $j\colon \F\longrightarrow \F_{T_{sin}}$ which extends to a finite morphism $j\colon \overline{\F}\longrightarrow \overline{\F}_{T_{sin}}$.
\end{prop}

Next we define the period domain and period map associated with the type $(T_{sin},T_{sym})$. Given a symmetry type $[(A, \lambda)]$, we denote $\overline{A}=A/\mu_3\subset \PSL(V)$. From previous discussion, the character $\lambda$ descends to $\overline{A}$ and is still denote by $\lambda$. We define $A^\lambda$ to be the kernel of the morphism $\CC^*\times \overline{A} \to \CC^*,\, (t, a)\mapsto t^{-2}\cdot \lambda(a)$. Let $F\in \calV_{\lambda}$, then $A^\lambda$ acts naturally on $S_F=\{y^2=F(x_0, x_1, x_2)\}$ by $(t, a)\cdot [y: x_0: x_1: x_2]=[ty: a [x_0: x_1: x_2]]$ \footnote{This is well-defined in the weighted projective space $\PP^{3,1,1,1}_{y, x_0, x_1, x_2}$.}. Suppose $F\in \Sigma_{T_{sin}, T_{sym}}$, then $A^{\lambda}$ also acts naturally on $\widetilde{S}_{F}$. Thus we have an action of $A^{\lambda}$ on $\Lambda$. Let $(\Lambda,M,H,\iota)$ be the data attached to $T_{sin}$, see \S\ref{section: period map}. Clearly, the involution $\iota$ on $\widetilde{S}_F$ is contained in $A^{\lambda}$.

The induced action of $A^{\lambda}$ on $\Lambda$ preserves $M$ and its orthogonal complement $M^{\perp}$ in $\Lambda$. Let $\eta$ be the character of $A^{\lambda}$ associated with $H^{2,0}(\widetilde{S}_F)$. Let $\Lambda_{\eta}$ be the eigenspace of $M^{\perp}_{\CC}$ associated with $\eta$.

When $\eta(A^{\lambda})\subset\{\pm 1\}$ (we then say $\eta$ is real), we define $\DD$ to be one component of $\PP\{x\in \Lambda_{\eta}\big{|}\varphi(x,x)=0, \varphi(x,\overline{x})>0\}$, which is the type $\IV$ domain associated with $\Lambda_{\eta}$. When $\eta$ is not real, we define $\DD\coloneqq\PP\{x\in\Lambda_{\eta}\big{|}\varphi(x,\overline{x})>0\}$, which is a complex hyperbolic ball.

Recall that $\Gamma_{T_{sin}}$ is the centralizer of $\iota$ in $\widehat{\Gamma}$. We write $r_i$ for the reflection with respect to $[E_i]\in \Lambda$. We have $r_i\in \Gamma_{T_{sin}}$ and $A^{\lambda}\subset \Gamma_{T_{sin}}$. Define $\Gamma_{T_{sin},T_{sym}}$ to be the normalizer of $\langle A^{\lambda}, r_1, \cdots, r_n\rangle$ in $\Gamma_{T_{sin}}$. We write $\Gamma=\Gamma_{T_{sin}, T_{sym}}$ for short. Then $\Gamma$ is an arithmetic group acting on $\DD$. Take $F\in \Sym^6(V^*)$ with $[F]\in \Sigma$. Choose an isomorphism
\begin{equation*}
\Phi\colon(H^2(\widetilde{S}_F),M_F,H, \iota, A^{\lambda})\longrightarrow (\Lambda,M,H,\iota,A^{\lambda}),
\end{equation*}
then $\Phi(H^{2,0}(\widetilde{S}_F))\in \DD$. Two choices of $\Phi$ differ by an element in the centralizer of $A^{\lambda}$ in $\Gamma_{T_{sin}}$, which is contained in $\Gamma$. Thus we obatin an analytic morphism $\Prd\colon \Sigma\longrightarrow \Gamma\bs\DD$. Taking quotient by $N$ on the left side, the map $\Prd$ descends to $\Prd\colon \F\longrightarrow \Gamma\bs \DD$. This morphism is called the period map for sextic curves of type $(T_{sin}, T_{sym})$. Let $\calH_*\coloneqq \DD\cap \calH_{\infty}$ be a hyperplane arrangement in $\DD$. It is clear that $\Prd(\F)\subset \Gamma\bs(\DD-\calH_*)$.
\begin{thm}
\label{theorem: last}
The period map $\Prd\colon \F\longrightarrow \Gamma\bs(\DD-\calH_*)$ is an algebraic open embedding, and extends to an isomorphism $\Prd\colon \overline{\F}\cong \overline{\Gamma\bs\DD}^{\calH_*}$.
\end{thm}
\begin{proof}
Take $F_1,F_2\in \Sym^6(V^*)$ such that $[F_1],[F_2]\in \Sigma$. Suppose $\Prd([F_1])=\Prd([F_2])$. Then there exist markings $\Phi_1,\Phi_2$ of $\widetilde{S}_{F_1},\widetilde{S}_{F_2}$ respectively, such that $\Phi_1(H^{2,0}(\widetilde{S}_{F_1}))=\Phi_2(H^{2,0}(\widetilde{S}_{F_2}))$. We have that $r_i\in \Gamma$ for any $1\le i\le n$. Hence we can assume without loss of generality that $\Phi_2^{-1}\Phi_1$ sends effective roots in $M_{F_1}$ to those in $M_{F_2}$. Then $\Phi_2^{-1}\Phi_1$ sends an ample class (say, $16H-[E_1]-\cdots-[E_n]$) to an ample class. By global Torelli for polarized $K3$ surfaces, there exists an isomorphism between $\widetilde{S}_{F_1}$ and $\widetilde{S}_{F_2}$ inducing $\Phi_2^{-1}\Phi_1\colon H^2(\widetilde{S}_{F_1},\ZZ)\longrightarrow H^2(\widetilde{S}_{F_2},\ZZ)$. This isomorphism is compatible with $\iota$, hence is induced by a linear transform $g\in \SL(V)$. We have $gAg^{-1}=A$. Now take any $a\in A$, we have
\begin{equation*}
\lambda(a)F_2=aF_2=agF_1=gaF_1=gag^{-1}F_2=\lambda(gag^{-1})F_2,
\end{equation*}
hence $\lambda(gag^{-1})=\lambda(a)$. This implies that $g\in N$. The injectivity of $\Prd$ follows.

Take $[F]\in \Sigma$ and $\Phi$ a marking of $\widetilde{S}_F$. Denote $x=\Phi(H^{2,0}(\widetilde{S}_F))\in \DD\subset\DD_{T_{sin}}$. By theorem \ref{theorem: main}, an open neighbourhood of $x$ in $\DD$ is induced by sextic curves of singular type $T_{sin}$. In particular, there exists a contractible open neighbourhood $U$ of $x$ in $\DD$, parametrizing a family of type $T_{sin}$ sextic curves $\mathscr{C}\longrightarrow U$. For every $x\in U$, let $Z(F_x)$ be the sextic curve over $x$. We have $\Phi(H^{2,0}(\widetilde{S}_{F_x}))=x$. Points in $U$ are polarized Hodge structures invariant under the action of $A^{\lambda}$. By theorem $1$ in \cite{burns1975torelli}, the $K3$ surface $\widetilde{S}_{F_x}$ admits an action of $A^{\lambda}$. Thus the family over $U$ of $K3$ surfaces $\widetilde{S}_{F_x}$ admits an action of $A^{\lambda}$. This action is induced by an action of $\overline{A}$ on $\mathscr{C}\longrightarrow U$. Therefore, The image of $\Prd\colon \F\longrightarrow \Gamma\bs\DD$ contains image of $U$ in $\Gamma\bs\DD$. Thus $\Prd$ has open image in $\Gamma\bs\DD$. We conclude that $\Prd$ is a bimeromorphism.

By \cite[Theorem A.13]{yu2018moduli}, there is a natural morphism $\Gamma\bs\DD\longrightarrow\widehat{\Gamma}\bs\widehat{\DD}$, which extends to $\overline{\Gamma\bs\DD}^{\calH_*}\longrightarrow \overline{\widehat{\Gamma}\bs\widehat{\DD}}^{\calH_{\infty}}$. Both are finite.

We have the following commutative diagram:
\begin{equation*}
\begin{tikzcd}
\F\arrow{r}{\Prd} \arrow{d}{j} &  \Gamma\bs\DD\arrow{d}{\pi}\\
\overline{\calM}\arrow{r}{\Prd} &  \overline{\widehat{\Gamma}\bs\widehat{\DD}}^{\calH_{\infty}}.
\end{tikzcd}
\end{equation*}
Since $\Prd(\F)$ is open in $\Gamma\bs\DD$, the two closures $\overline{j(\F)}$ and $\overline{\pi(\Gamma\bs\DD)}$ are identified via $\Prd\colon \overline{\calM}\cong \overline{\widehat{\Gamma}\bs\widehat{\DD}}^{\calH_{\infty}}$. By lemma 5.4 in \cite{yu2018moduli}, there is an isomorphism $\overline{\F}\cong\overline{\Gamma\bs\DD}^{\calH_*}$ extending $\Prd\colon \F\longrightarrow \Gamma\bs(\DD-\calH_*)$.
\end{proof}

As an end, we give two examples of theorem \ref{theorem: last}.

\begin{ex}
\label{example: six involution}
An involution $\tau$ of $\PP(V)$ must be represented by $\diag(1,1,-1)$ on $V$, with respect to certain choice of coordinate $(x_1, x_2, x_3)$. Let $\overline{A}=\{id, \tau\}$ and $A$ be the preimage of $\overline{A}$ in $\SL(V)$. Take $\lambda$ to be the trivial character of $A$. We take symmetry type $T_{sym}=[(A,\lambda)]$, and singular type $T_{sin}=(1,1,1,1,1,1)$. We next describe $\Sigma_{T_{sin},T_{sym}}$. A point in $\Sigma_{T_{sin}, T_{sym}}$ corresponds to six lines on $\PP(V)$ preserved by the action of $\tau$, such that any three of them do not intersect at one point. A line $l\subset \PP(V)$ satisfies $\tau(l)=l$ if and only if $l=\{x_3=0\}$ or $l$ passes through $[0:0:1]$. We call a pair of different lines $\tau$-conjugate, if one line is mapped to the other via $\tau$. The points in $\Sigma_{T_{sin},T_{sym}}$ belong to the following two irreducible components: one component $\Sigma_1$ with the corresponding six lines being union of three $\tau$-conjugate pairs, the other component $\Sigma_2$ with the corresponding six lines being union of two different lines passing through $[0:0:1]$ and two $\tau$-conjugate pairs. Denote by $\PP(V^*)$ the space of lines on $\PP(V)$, and by $\PP(V^*)/\tau\cong \PP(1,1,2)$ the space of $\tau$-orbits in $\PP(V^*)$. Let $N=\SL(\GL(2,\CC)\times\CC^{\times})$.

Consider the morphism
\begin{equation*}
j_1\colon S_3\bs(\PP(V^*)/\tau)^3\longrightarrow S_6\bs\PP(V^*)^6
\end{equation*}
sending $[(l_1,l_2,l_3)]$ to $[(l_1,l_2,l_3,\tau(l_1),\tau(l_2), \tau(l_3))]$ for $l_1,l_2,l_3\in\PP(V^*)$. This morphism is a closed embedding with image $\overline{\Sigma}_1$. Hence $\widetilde{\Sigma}_1=\overline{\Sigma}_1=S_3\bs(\PP(V^*)/\tau)^3$. So the $\GIT$ quotient is \begin{equation*}
\overline{\F}_1\cong (N \times S_3)\dbs ((\PP(V^*)/\tau)^3, \calO(1)\boxtimes\calO(1)\boxtimes\calO(1))\cong (N\times S_3)\dbs (\PP(1,1,2)^3,\calO(1)\boxtimes\calO(1)\boxtimes\calO(1))
\end{equation*}
and is isomorphic to the Baily-Borel compactification of an arithmetic quotient of a type $\IV$ domain with dimension $2$, i.e. product of two upper half planes by theorem \ref{theorem: last}.

Let $V_2$ be the quotient of $V$ by the line $\{x_1=x_2=0\}\subset V$. Then $\PP(V_2^*)$ parametrizes lines in $\PP(V)$ passing through $[0:0:1]$. Consider the finite morphism
\begin{equation*}
j_2\colon (S_2\bs \PP(V_2^*)^2)\times (S_2\bs(\PP(V^*)/\tau)^2)\longrightarrow S_6\bs\PP(V^*)^6
\end{equation*}
sending $[(l_1,l_2,l_3,l_4)]$ to $[(l_1,l_2,l_3,l_4, \tau(l_3),\tau(l_4))]$ for $l_1,l_2$ passing through $[0:0:1]$ and $l_3, l_4\in\PP(V^*)$. This morphism is generically injective with image $\overline{\Sigma}_1$. Hence $\widetilde{\Sigma}_1=(S_2\bs \PP(V_2^*)^2)\times (S_2\bs(\PP(V^*)/\tau)^2)$. So the $\GIT$-quotient is $\overline{\F}_2\cong N \times (S_2)^4\dbs (\PP(V_2^*)^2\times\PP(V^*)^2, \calO(1)^{\boxtimes 4})$, and is isomorphic to the Baily-Borel compactification of an arithmetic quotient of a type $\IV$ domain with dimension $2$, i.e. product of two upper half planes, by theorem \ref{theorem: last}.
\end{ex}

\begin{ex}
\label{example: line quintic three}
Let $\rho$ be an automorphism of $\PP(V)$ represented by $\diag(1,1,\omega)$ with respect to certain choice of coordinate $(x_1,x_2,x_3)$. Here $\omega=e^{\frac{2\pi\sqrt{-1}}{3}}$ is a third root of unity. Let $\overline{A}=\{1,\rho,\rho^2\}$ and $A$ be the preimage of $\overline{A}$ in $\SL(V)$. Take $\lambda$ to be the trivial character of $A$. We take symmetry type $T_{sym}=[(A,\lambda)]$ and singular type $T_{sin}=(1,5)$. In this case $\Sigma_{T_{sin},T_{sym}}$ is irreducible and we write $\Sigma=\Sigma_{T_{sin},T_{sym}}$ for short. The equation of a point in $\Sigma$ can be arranged as
\begin{equation*}
(t_1 x_1+t_2 x_2)(L_2(x_1,x_2)x_3^3+L_5(x_1,x_2)),
\end{equation*}
where $t_1,t_2\in \CC$ and $L_2,L_5$ are polynomials of $x_1,x_2$ of degree $2,5$ respectively. The line $(t_1 x_1+t_2 x_2=0)$ passes through $[0:0:1]$, and such lines are parametrized by $\PP(V_2^*)$ as defined in example \ref{example: six involution}. The quintics $Z(L_2(x_1,x_2)x_3^3+L_5(x_1,x_2))$ are parametrized by $\PP(\Sym^2(V_2^*)\oplus \Sym^5(V_2^*))$. Consider the natural finite morphism
\begin{equation*}
j\colon \PP(V_2^*)\times\PP(\Sym^2(V_2^*)\oplus \Sym^5(V_2^*))\longrightarrow \PP(V^*)\times\PP(\Sym^5(V^*)).
\end{equation*}
This morphism is generically injective with image $\overline{\Sigma}$. Thus $\widetilde{\Sigma}=\PP(V_2^*)\times\PP(\Sym^2(V_2^*)\oplus \Sym^5(V_2^*))$. The reductive group $N$ is equal to $\SL(\GL(2,\CC)\times \CC^{\times})$. So the $\GIT$-quotient is 
\begin{equation}
\label{example: order 3 GIT}
\overline{\F}\cong N \dbs (\PP(V_2^*)\times\PP(\Sym^2(V_2^*)\oplus \Sym^5(V_2^*)), \calO(1)\boxtimes \calO(1)). 
\end{equation}

Consider the following short exact sequence:
\begin{equation*}
1\longrightarrow \ZZ/2\longrightarrow \SL(2, \CC)\times \CC^*\longrightarrow N\longrightarrow 1
\end{equation*}
where $\diag(B, t)\in \SL(2, \CC)\times \CC^*$ is sent to $\diag(tB, t^{-2})\in N$. The kernel $\ZZ/2$ is acting trivially on $\widetilde{\Sigma}$. Straightforward calculation shows that the $\GIT$ quotient of $(\PP(V_2^*)\times\PP(\Sym^2(V_2^*)\oplus \Sym^5(V_2^*)), \calO(1)\boxtimes \calO(1))$ by $\CC^*$ ($t\in \CC^*$ is acting via $\diag(t,t,t^{-2})$) is $\PP(V_2^*)\times\PP(\Sym^2(V_2^*))\times\PP( \Sym^5(V_2^*))$ with polarization $\calO(3)\boxtimes \calO(2)\boxtimes\calO(1)$. So the $\GIT$ quotient \eqref{example: order 3 GIT} can be rearranged as 
\begin{equation*}
\overline{\F}\cong \SL(2)\dbs (\PP^1\times \PP^2\times \PP^5, \calO(3)\boxtimes \calO(2)\boxtimes\calO(1)).
\end{equation*}
In this case, the character of $\mu_3$ on $H^{2,0}(\widetilde{S}_F)$ is non-real. Thus $\overline{\F}$ is isomorphic to the Baily-Borel compactification of an arithmetic ball quotient with dimension $5$ by theorem \ref{theorem: last}.

Notice that $S_n\bs((\PP^1)^n, \calO(1)^{\boxtimes n})\cong (\PP^n, \calO(1))$. So $\overline{\F}$ is a quotient of
\begin{equation*}
\widetilde{\F}\cong  \SL(2)\dbs ((\PP^1)^8, \calO(3)\boxtimes \calO(2)^{\boxtimes 2}\boxtimes \calO(1)^{\boxtimes 5}).
\end{equation*}
by permutation group $S_2\times S_5$. The $\GIT$ quotient $\widetilde{\F}$ also appears in Deligne-Mostow's list of ball quotients. See example (11) in \cite{mostow1988ball} or example (20) in \cite{thurston1998shapes}. The construction in Deligne-Mostow \cite{deligne1986monodromy} uses period of certain cyclic cover of $\PP^1$, which is quite different from period of $K3$ surface used here. We conjecture these two constructions give the same ball-quotient structure on $\overline{\F}$, and there exists a geometric construction relating the two approaches. 
\end{ex}

\bibliography{reference}

\providecommand{\bysame}{\leavevmode\hbox to3em{\hrulefill}\thinspace}
\providecommand{\MR}{\relax\ifhmode\unskip\space\fi MR }
\providecommand{\MRhref}[2]{%
  \href{http://www.ams.org/mathscinet-getitem?mr=#1}{#2}
}
\providecommand{\href}[2]{#2}
\begin{thebibliography}{GMGZ18}

\bibitem[ACT02]{allcock2002complex}
D.~Allcock, J.~A. Carlson, and D.~Toledo, \emph{The complex hyperbolic geometry
  of the moduli space of cubic surfaces}, J. Algebraic Geom. \textbf{11}
  (2002), no.~4, 659--724.

\bibitem[ACT11]{allcock2011moduli}
\bysame, \emph{The moduli space of cubic threefolds as a ball quotient}, Mem.
  Amer. Math. Soc. \textbf{209} (2011), no.~985, xii+70.

\bibitem[Ap{\'e}79]{apery1979irrationalite}
R.~Ap{\'e}ry, \emph{Irrationalit{\'e} de $\zeta$(2) et $\zeta$(3)},
  Ast{\'e}risque \textbf{61} (1979), no.~11-13, 1.

\bibitem[BB66]{borel1966}
Jr. W.~L. Baily and A.~Borel, \emph{Compactification of arithmetic quotients of
  bounded symmetric domains}, Ann. of Math. (2) \textbf{84} (1966), 442--528.

\bibitem[BP84]{beukers1984afamily}
F.~Beukers and C.A.M. Peters, \emph{A family of {K3} surfaces and $\zeta$(3)},
  Journal für die reine und angewandte Mathematik \textbf{351} (1984), 42--54.

\bibitem[BR75]{burns1975torelli}
D.~Burns and M.~Rapoport, \emph{On the {T}orelli problem for {K}\"ahlerian
  {$K3$} surfaces}, Ann. Sci. \'Ecole Norm. Sup. (4) \textbf{8} (1975), no.~2,
  235--273.

\bibitem[Cay71]{cayley1869/71sextic}
A.~Cayley, \emph{A second memoir on quartic surfaces}, Proc. Lond. Math. Soc.
  (1869/71).

\bibitem[DM86]{deligne1986monodromy}
P.~Deligne and G.~D. Mostow, \emph{Monodromy of hypergeometric functions and
  nonlattice integral monodromy}, Inst. Hautes \'Etudes Sci. Publ. Math.
  (1986), no.~63, 5--89.

\bibitem[Dol96]{dolgachev1996mirror}
I.~V. Dolgachev, \emph{Mirror symmetry for lattice polarized {$K3$} surfaces},
  J. Math. Sci. \textbf{81} (1996), no.~3, 2599--2630, Algebraic geometry, 4.

\bibitem[Gal09]{galati2009cusps}
C.~Galati, \emph{On the number of moduli of plane sextics with six cusps}, Ann.
  Mat. Pura Appl. (4) \textbf{188} (2009), no.~2, 359--368.

\bibitem[GMGZ18]{gallardo2017compactifications}
P.~Gallardo, J.~Martinez-Garcia, and Z.~Zhang, \emph{Compactifications of the
  moduli space of plane quartics and two lines}, Eur. J. Math. \textbf{4}
  (2018), no.~3, 1000–1034.

\bibitem[Har86]{harris1986severi}
Joe Harris, \emph{On the {S}everi problem}, Invent. Math. \textbf{84} (1986),
  no.~3, 445--461.

\bibitem[Hun00]{hunt2000nice}
B.~Hunt, \emph{Nice modular varieties}, Experiment. Math. \textbf{9} (2000),
  no.~4, 613--622.

\bibitem[Kon00]{kondo2000complex}
S.~Kond\=o, \emph{A complex hyperbolic structure for the moduli space of curves
  of genus three}, J. Reine Angew. Math. \textbf{525} (2000), 219--232.

\bibitem[Kon02]{kondo2000moduli}
\bysame, \emph{The moduli space of curves of genus 4 and {D}eligne-{M}ostow's
  complex reflection groups}, Algebraic geometry 2000, {A}zumino ({H}otaka),
  Adv. Stud. Pure Math., vol.~36, Math. Soc. Japan, Tokyo, 2002, pp.~383--400.

\bibitem[Laz09]{laza2009deformation}
R.~Laza, \emph{Deformations of singularities and variation of {GIT} quotients},
  Trans. Amer. Math. Soc. \textbf{361} (2009), no.~4, 2109--2161.

\bibitem[Laz16]{laza2016persepectives}
\bysame, \emph{Perspectives on the construction and compactification of moduli
  spaces}, Compactifying moduli spaces, Adv. Courses Math. CRM Barcelona,
  Birkh\"auser/Springer, Basel, 2016, pp.~1--39.

\bibitem[Loo03a]{looijenga2003ball}
E.~Looijenga, \emph{Compactifications defined by arrangements. {I}. {T}he ball
  quotient case}, Duke Math. J. \textbf{118} (2003), no.~1, 151--187.

\bibitem[Loo03b]{looijenga2003typefour}
\bysame, \emph{Compactifications defined by arrangements. {II}. {L}ocally
  symmetric varieties of type {IV}}, Duke Math. J. \textbf{119} (2003), no.~3,
  527--588.

\bibitem[LPZ18]{laza2017moduli}
R.~Laza, G.~Pearlstein, and Z.~Zhang, \emph{On the moduli space of pairs
  consisting of a cubic threefold and a hyperplane}, Adv. Math. \textbf{340}
  (2018), 684--722.

\bibitem[LS07]{looijenga2007period}
E.~Looijenga and R.~Swierstra, \emph{The period map for cubic threefolds},
  Compos. Math. \textbf{143} (2007), no.~4, 1037--1049.

\bibitem[Lun75]{luna1975Adh}
D.~Luna, \emph{Adh\'erences d'orbite et invariants}, Invent. Math. \textbf{29}
  (1975), no.~3, 231--238.

\bibitem[LW95]{lee1995doublebranchedcover}
R.~Lee and S.~H. Weintraub, \emph{On the homology of double branched covers},
  Proc. Amer. Math. Soc. \textbf{123} (1995), no.~4, 1263--1266.

\bibitem[MFK94]{mumford1994geometric}
D.~Mumford, J.~Fogarty, and F.~Kirwan, \emph{Geometric invariant theory}, third
  ed., Ergebnisse der Mathematik und ihrer Grenzgebiete (2) [Results in
  Mathematics and Related Areas (2)], vol.~34, Springer-Verlag, Berlin, 1994.
  \MR{1304906}

\bibitem[Mos88]{mostow1988ball}
G.~D. Mostow, \emph{On discontinuous action of monodromy groups on the complex
  {$n$}-ball}, J. Amer. Math. Soc. \textbf{1} (1988), no.~3, 555--586.

\bibitem[MSY92]{matsumoto1992sixlines}
K.~Matsumoto, T.~Sasaki, and M.~Yoshida, \emph{The monodromy of the period map
  of a {$4$}-parameter family of {$K3$} surfaces and the hypergeometric
  function of type {$(3,6)$}}, Internat. J. Math. \textbf{3} (1992), no.~1,
  164.

\bibitem[PZ19]{pearlstein2019horikawa}
G.~Pearlstein and Z.~Zhang, \emph{A generic global {T}orelli theorem for
  certain {H}orikawa surfaces}, Algebr. Geom. \textbf{6} (2019), no.~2,
  132--147.

\bibitem[Ser06]{sernesi2006deformation}
E.~Sernesi, \emph{Deformations of algebraic schemes}, Grundlehren der
  Mathematischen Wissenschaften [Fundamental Principles of Mathematical
  Sciences], vol. 334, Springer-Verlag, Berlin, 2006.

\bibitem[Sev59]{severi1959nodalcurve}
F.~Severi, \emph{Geometria dei sistemi algebrici sopra una superficie e sopra
  una variet\`a algebrica. {V}ols. 2, 3}, Volumi secondo e terzo in
  continuazione del volume primo dello stesso autore che porta il titolo:
  Serie, sistemi d'equivalenza e corrispondenze algebriche sulle variet\`a
  algebriche, Edizioni Cremonese, Rome. Vol. 2, 1959.

\bibitem[Sev68]{severi1968curve}
\bysame, \emph{Vorlesungen \"uber algebraische {G}eometrie: {G}eometrie auf
  einer {K}urve, {R}iemannsche {F}l\"achen, {A}belsche {I}ntegrale},
  Berechtigte Deutsche \"Ubersetzung von Eugen L\"offler. Mit einem
  Einf\"uhrungswort von A. Brill. Begleitwort zum Neudruck von Beniamino Segre.
  Bibliotheca Mathematica Teubneriana, Band 32, Johnson Reprint Corp., New
  York-London, 1968.

\bibitem[Sha80]{shah1980completek3}
J.~Shah, \emph{A complete moduli space for {$K3$}\ surfaces of degree {$2$}},
  Ann. of Math. (2) \textbf{112} (1980), no.~3, 485--510.

\bibitem[Thu98]{thurston1998shapes}
W.~P. Thurston, \emph{Shapes of polyhedra and triangulations of the sphere},
  The {E}pstein birthday schrift, Geom. Topol. Monogr., vol.~1, Geom. Topol.
  Publ., Coventry, 1998, pp.~511--549.

\bibitem[YZ20]{yu2018moduli}
C.~Yu and Z.~Zheng, \emph{Moduli spaces of symmetric cubic fourfolds and
  locally symmetric varieties}, Algebra \& Number Theory \textbf{14} (2020),
  no.~10, 2647--2683.

\bibitem[Zar83]{zarhin1983hodge}
Y.~Zarhin, \emph{Hodge groups of {$K3$} surfaces}, J. Reine Angew. Math.
  \textbf{341} (1983), 193--220.

\end{thebibliography}
\bibliographystyle{alpha}

\Addresses
\end{document}